\documentclass[11pt,a4paper]{amsart}
\usepackage{fullpage}

\usepackage{utopia}

\usepackage{amssymb,amsmath,amstext,amsthm,bbm,dsfont,nicefrac,color,tikz}
\usepackage{caption}
\usepackage{float}
\usepackage{verbatim,fancyvrb}
\usepackage{enumerate}
\usepackage{longtable,array,booktabs,makecell}

\usepackage[all,color]{xy}
\xyoption{dvips}
\xyoption{color}

\usepackage{aliascnt}
\usepackage[colorlinks=true,urlcolor=cyan,linkcolor=blue,citecolor=magenta]{hyperref}

\makeatletter
\def\newaliasedtheorem#1[#2]#3{%
  \newaliascnt{#1@alt}{#2}
  \newtheorem{#1}[#1@alt]{#3}
  \expandafter\newcommand\csname #1@altname\endcsname{#3}
}
\makeatother

\theoremstyle{theorem}
\newtheorem{theorem}{Theorem}
\newaliasedtheorem{proposition}[theorem]{Proposition}
\newaliasedtheorem{lemma}[theorem]{Lemma}
\newaliasedtheorem{corollary}[theorem]{Corollary}
\newaliasedtheorem{question}[theorem]{Question}

\theoremstyle{definition}
\newaliasedtheorem{definition}[theorem]{Definition}
\newaliasedtheorem{remark}[theorem]{Remark}
\newaliasedtheorem{construction}[theorem]{Construction}

\newcommand{\cS}{{\mathcal S}}

\newcommand{\cG}{{\mathcal G}}

\newcommand{\F}{{\mathbbm F}}

\newcommand{\Hom}{\operatorname{Hom}}
\newcommand{\aut}{\operatorname{Aut}}

\newcommand{\psl}{\operatorname{PSL}}
\newcommand{\psil}{\operatorname{P\Sigma L}}
\newcommand{\SL}{\operatorname{SL}}

\newcommand{\set}[2]{\left\{#1\ \middle\vert\ #2\right\}}
\newcommand{\size}[1]{\left| #1 \right|}

\newcommand{\bigslant}[2]{{\raisebox{.2em}{$#1$}\left/\raisebox{-.2em}{$#2$}\right.}}

\newcommand{\lra}{\longrightarrow}

\newenvironment{psmmatrix}
  {\scriptsize\begin{pmatrix}}
  {\end{pmatrix}}

\newcommand{\pbibd}{\mathrm{PBIBD}}
\newcommand{\gpbibd}{\mathrm{GPBIBD}}

\title{On an infinite family of generalized PBIBDs}

\author{Daniel Kalmanovich}
\address{Einstein Institute of Mathematics\\ The Hebrew University of Jerusalem}
\email{daniel.kalmanovich@gmail.com}
\date{August 13, 2015}

\begin{document}

\maketitle


\begin{abstract}
	We consider a generalization of the notion of partially balanced incomplete block designs (PBIBDs), by relaxing the requirement that the underlying association scheme be commutative. An infinite family of such generalizations is constructed, one for each prime power $q$ congruent to $1$ modulo $4$.
\end{abstract}

\section{Introduction and Preliminaries}
Combinatorial designs are fascinating objects which have been studied for centuries, and have origins in diverse areas (cf. the introduction of~\cite{ColbD07} for a brief historical account). This paper focuses on a particular kind of combinatorial design, originating in the theory of the design of experiments, called partially balanced incomplete block design (PBIBD). Formally introduced by Bose and Nair in~\cite{BoseN39}, a PBIBD is a certain incidence structure together with an underlying symmetric association scheme\footnote{There are longstanding terminology disagreements in this area. In this paper {\em association scheme} means {\em homogeneous coherent configuration}.} which provides many regularity properties to the whole structure. We refer to Bailey's book~\cite{Bail04} which gives a statistician-friendly account of association schemes and their use in the design of experiments.

Association schemes arose naturally in algebraic combinatorics, independent of any applications to statistics, in works pertaining to permutation groups and, independently, to the graph isomorphism problem. We refer to Cameron's survey~\cite{Came03} for a brief overview of the permutation groups aspect, and the celebrated monograph~\cite{BrouCN89} for the many graph theoretic aspects of association schemes. 

Natural generalizations of the notion of a PBIBD were suggested and studied over the years. Among these, notions that take a more general underlying structure by relaxing the various requirements in the definition of an association scheme were introduced and studied in~\cite{Shah59} and in~\cite{Sinh82}. In~\cite{Nair64}, Nair suggested a generalization of PBIBDs to incidence structures with underlying association schemes which are not necessarily commutative.

In this paper we present a construction of an infinite family of such generalizations of PBIBDs, one for each prime power $q$ congruent to $1$ modulo $4$. This construction stemmed from the analysis of designs constructed by Nevo in~\cite{Nevo15}.

We assume that the reader is familiar with basics of association scheme theory, but for notational purposes we recall the definitions of an association scheme and a $\pbibd$.

\begin{definition}
An {\em association scheme with $m$ associate classes} is a finite set $X$ together with $m+1$ binary relations $R_i\subseteq X^2$, $0\leq i\leq m$, such that:
\begin{itemize}
\item $R_0=\set{(x,x)}{x\in X}$,
\item for each relation $R_i$, the inverse relation $R_i^{-1}=\set{(y,x)}{(x,y)\in R_i}$ is also an associate class,
\item $\bigcup\limits_{i=0}^d R_i=X^2$, and $R_i\cap R_j=\emptyset$ whenever $i\neq j$,
\item for any pair $(x,z)\in R_k$, the number of elements $y\in X$ such that $(x,y)\in R_i$ and $(y,z)\in R_j$ depends only on $i,j,k$. This number is denoted by $p_{ij}^k$.
\end{itemize}
The numbers in the last item above are called the {\em intersection numbers} of the association scheme. We will also use the term {\em $i$-associates} when referring to a pair $(x,y)\in R_i$.
\end{definition}

A {\em symmetric} association scheme is one in which all the binary relations are symmetric.

\begin{definition}\label{def:PBIBD}
A {\em partially balanced incomplete block design with $m$ associate classes} denoted $\pbibd(m)$ is a block design $D=(P,B)$ with $|P|=v$ points, $|B|=b$ blocks, where each block has size $k$, each point lies on $r$ blocks, and such that there is a symmetric association scheme defined on $P$ where, if $(x,y)$ are $j$-th associates then they occur together in precisely $\lambda_j$ blocks for $1\leq j\leq m$. The numbers $v,b,r,k,\lambda_j (1\leq j\leq m)$ are called the {\em parameters} of the $\pbibd(m)$.
\end{definition}

Our definition of a {\em generalized} PBIBD is verbatim \autoref{def:PBIBD} dropping the word symmetric. It follows that if $R_j$ is the inverse relation of $R_i$ then $\lambda_j=\lambda_i$.

A special feature of our family of $\gpbibd$s is that the number of non-vanishing values of the $\lambda_j$s is constant: we have just the value $5$ when $q=5^{\alpha}$, and the values $4$ and $1$ when $q$ is not a power of $5$. Moreover, there is exactly one class for each non-zero value.

\section{The construction}
Let $q=p^{\alpha}$ be a prime power congruent to $1$ modulo $4$, and consider the vector space $V={\F_q}^2$, where $\F_q$ denotes the finite field with $q$ elements. Let $\omega$ be a generator of ${\F_q}^{\times}$, and denote $i=\omega^{\frac{q-1}{4}}$. Our set of points is $P=\bigslant{V\setminus\{0\}}{\langle i\rangle}$, that is, we identify the four vectors $\pi$, $-\pi$, $i\pi$ and $-i\pi$, for all non-zero vectors $\pi\in V$, to obtain the set $P$ of $\frac{q^2-1}{4}$ points. Let $G=\psl(2,q)$, the quotient group of the matrix group $\SL(2,q)=\set{\begin{psmmatrix}a & b\\ c & d\end{psmmatrix}}{ad-bc=1}$ by its center $Z(\SL(2,q))=\left\{\begin{psmmatrix}1 & 0\\ 0 & 1\end{psmmatrix},\begin{psmmatrix}-1 & 0\\ 0 & -1\end{psmmatrix}\right\}$. Each element of $G$ is a set of $2$ matrices, and consider the natural action of $G$ on $P$ via multiplication (of representatives): 
$$\begin{psmmatrix}a & b\\ c & d\end{psmmatrix}\cdot \begin{psmmatrix}x\\ y\end{psmmatrix}=\begin{psmmatrix}ax+by\\ cx+dy\end{psmmatrix},$$ 
here $\begin{psmmatrix}a & b\\ c & d\end{psmmatrix}$ is a representative of the element $\left\{\begin{psmmatrix}a & b\\ c & d\end{psmmatrix},\begin{psmmatrix}-a & -b\\ -c & -d\end{psmmatrix}\right\}\in G$, and $\begin{psmmatrix}x\\ y\end{psmmatrix}$ is a representative of the element 
$\left\{\begin{psmmatrix}x\\ y\end{psmmatrix},\begin{psmmatrix}-x\\ -y\end{psmmatrix},\begin{psmmatrix}ix\\ iy\end{psmmatrix},\begin{psmmatrix}-ix\\ -iy\end{psmmatrix}\right\}\in P$. Define 
$$T=\left\{\begin{psmmatrix}1\\ 0\end{psmmatrix},\begin{psmmatrix}0\\ 1\end{psmmatrix},\begin{psmmatrix}1\\ 1\end{psmmatrix},\begin{psmmatrix}1+i\\ 1\end{psmmatrix},\begin{psmmatrix}i\\ 1\end{psmmatrix},\begin{psmmatrix}1\\ 1-i\end{psmmatrix}\right\},$$ and denote $B=T^G$, the orbit of $T$ under the natural action of $G$ (the induced action on $6$-sets). We have thus defined an incidence structure $D=(P,B)$ on which $G$ acts transitively both on the points and on the blocks. In fact, if we analyse the action of $G$ on pairs of points $P$ we can give a geometric interpretation to the basic block $T$, and thus to all blocks of $B$. Consider the points of $T$ as the vertices of a regular octahedron, with edges as depicted in the following figure:



\begin{center}
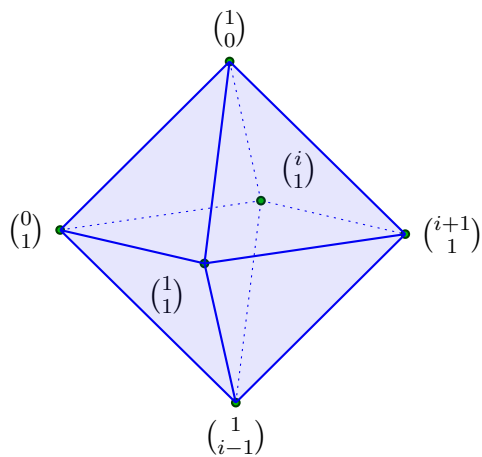
\begin{figure}[H]
\begin{tikzpicture}%
[x={(0.986706cm, -0.011933cm)},
y={(-0.017755cm, 0.983416cm)},
z={(-0.161542cm, -0.180970cm)},
scale=2.3,
back/.style={dotted, thin},
edge/.style={color=blue!95!black, thick},
facet/.style={fill=blue!95!black,fill opacity=0.1},
vertex/.style={inner sep=1pt,circle,draw=green!25!black,fill=green!75!black,thick,anchor=base}]
%
%
\coordinate (a) at (-1.00000, 0.00000, 0.00000);
\coordinate (b) at (0.00000, -1.00000, 0.00000);
\coordinate (c) at (0.00000, 0.00000, -1.00000);
\coordinate (d) at (0.00000, 0.00000, 1.00000);
\coordinate (e) at (0.00000, 1.00000, 0.00000);
\coordinate (f) at (1.00000, 0.00000, 0.00000);
\draw[edge,back] (-1.00000, 0.00000, 0.00000) -- (0.00000, 0.00000, -1.00000);
\draw[edge,back] (0.00000, -1.00000, 0.00000) -- (0.00000, 0.00000, -1.00000);
\draw[edge,back] (0.00000, 0.00000, -1.00000) -- (0.00000, 1.00000, 0.00000);
\draw[edge,back] (0.00000, 0.00000, -1.00000) -- (1.00000, 0.00000, 0.00000);
\node[vertex,label=left:{$\binom{0}{1}$}] at (a)     {};
\node[vertex,label=below:{$\binom{1}{i-1}$}] at (b)     {};
\node[vertex,label={[label distance=0.05cm]20:$\binom{i}{1}$}] at (c)     {};
\node[vertex,label={[label distance=0.05cm]200:$\binom{1}{1}$}] at (d)     {};
\node[vertex,label=above:{$\binom{1}{0}$}] at (e)     {};
\node[vertex,label=right:{$\binom{i+1}{1}$}] at (f)     {};

\fill[facet] (0.00000, 1.00000, 0.00000) -- (-1.00000, 0.00000, 0.00000) -- (0.00000, 0.00000, 1.00000) -- cycle {};
\fill[facet] (0.00000, 0.00000, 1.00000) -- (-1.00000, 0.00000, 0.00000) -- (0.00000, -1.00000, 0.00000) -- cycle {};
\fill[facet] (1.00000, 0.00000, 0.00000) -- (0.00000, 0.00000, 1.00000) -- (0.00000, 1.00000, 0.00000) -- cycle {};
\fill[facet] (1.00000, 0.00000, 0.00000) -- (0.00000, -1.00000, 0.00000) -- (0.00000, 0.00000, 1.00000) -- cycle {};
\draw[edge] (-1.00000, 0.00000, 0.00000) -- (0.00000, -1.00000, 0.00000);
\draw[edge] (-1.00000, 0.00000, 0.00000) -- (0.00000, 0.00000, 1.00000);
\draw[edge] (-1.00000, 0.00000, 0.00000) -- (0.00000, 1.00000, 0.00000);
\draw[edge] (0.00000, -1.00000, 0.00000) -- (0.00000, 0.00000, 1.00000);
\draw[edge] (0.00000, -1.00000, 0.00000) -- (1.00000, 0.00000, 0.00000);
\draw[edge] (0.00000, 0.00000, 1.00000) -- (0.00000, 1.00000, 0.00000);
\draw[edge] (0.00000, 0.00000, 1.00000) -- (1.00000, 0.00000, 0.00000);
\draw[edge] (0.00000, 1.00000, 0.00000) -- (1.00000, 0.00000, 0.00000);
\end{tikzpicture}
\caption{The basic octahedron $T$}
\end{figure}
\end{center}

With this description at hand we will freely switch between the terms {\em block} and {\em octahedron} when we refer to elements of $B$. Also, we will use the terms {\em edges} and {\em diagonals} for pairs of points of $D$ which form an edge in an octahedron, and pairs of points of $D$ which form a diagonal in an octahedron (a pair of antipodal points), respectively. Of course, any other pair of points do not lie in a block together.

\begin{lemma}
\begin{enumerate}
\item $G$ acts transitively on the edges of $D$.
\item $G$ acts transitively on the diagonals of $D$.
\end{enumerate} 
\end{lemma}

Let us denote by $\cS$ the Schurian association scheme of the action of $G$ on $P$. Then we have

\begin{corollary}
The block design $D=(P,B)$ is a $\gpbibd$ with underlying association scheme $\cS$.
\end{corollary}

\begin{lemma}\label{lem:pstab}
The stabilizer $G_{\pi}$ of a point $\pi\in P$ in $G$ is isomorphic to ${C_p}^{\alpha}\rtimes C_2$, where $q=p^{\alpha}$.
\end{lemma}
\begin{proof}
Since $G$ acts transitively on $P$ it suffices to consider the stabilizer $G_{\binom{1}{0}}$ of $\begin{psmmatrix}1\\ 0\end{psmmatrix}$. So let $\begin{psmmatrix}a & b\\ c & d\end{psmmatrix}\in G_{\binom{1}{0}}$, that is
$$
\begin{psmmatrix}a & b\\ c & d\end{psmmatrix}\cdot \begin{psmmatrix}1\\0\end{psmmatrix}= \begin{psmmatrix}1\\0\end{psmmatrix}.
$$
It follows that $a=1$ or $i$, and $c=0$, thus
$$
\begin{psmmatrix}a & b\\ c & d\end{psmmatrix}=
\begin{psmmatrix}1 & b\\ 0 & d\end{psmmatrix}\quad\text{or}\quad
\begin{psmmatrix}i & b\\ 0 & d\end{psmmatrix}
$$
the determinant $=1$ condition implies that $d=1$ in the first case, and $d=-i$ in the second case, so we have
$$
\begin{psmmatrix}a & b\\ c & d\end{psmmatrix}=
\begin{psmmatrix}1 & b\\ 0 & 1\end{psmmatrix}\quad\text{or}\quad
\begin{psmmatrix}i & b\\ 0 & -i\end{psmmatrix}.
$$
Therefore
$$
G_{\binom{1}{0}}=\set{\begin{psmmatrix}u & x\\ 0 & u^{-1}\end{psmmatrix}}{u\in\{1,i\},x\in\F_q}.
$$
\end{proof}

The cases $q=5^{\alpha}$ and $q=p^{\alpha}$, $p\neq 5$ are different in an essential way, so we analyse them separately. The reason for the difference between the case of characteristic $5$ and other characteristics is the fact that in the case of characteristic $5$, the subgroup $\{1,-1,i,-i\}$ forms the multiplicative group of the prime subfield $\F_5$ (and not just a subgroup of the multiplicative group of the field). A manifestation of this fact that will be useful for our purposes later is formulated in the following lemma.

\begin{lemma}\label{lem:pdiff5}
$1+i=-i$ if and only if $q=5^{\alpha}$.
\end{lemma}
\begin{proof}
$$
1+i=-i\qquad\Longleftrightarrow\qquad 2i=-1\qquad\Longleftrightarrow\qquad 2i=i^2\qquad\Longleftrightarrow\qquad i=2
$$
so we have $-1=i^2=2^2=4$, that is, $\F_q$ is a field of characteristic $5$.
\end{proof}

\section*{The case $p\neq 5$}

In this case the stabilizer $G_T$ of the basic block $T$ in $G$ is a group of rotations of the regular octahedron. It is in fact half of the full group of rotations of the regular octahedron. There are two types of rotations of a regular octahedron:
\begin{enumerate}[(a)]
\item rotation about the axis defined by a pair of antipodal points, and
\item rotation about the axis formed by connecting the centers of a pair of opposite faces.
\end{enumerate}
In the full group of rotations of the octahedron, the rotations of type (a) are of order $4$, and the rotations of type (b) are of order $3$ thus producing a group of order $24$. In our group $G_T$ we take the square of rotations of type (a) and all rotations of type (b), obtaining a group of order $12$. In this description $G_T=\langle g,f\rangle$ where 
\begin{itemize}
\item $g=\begin{psmmatrix} -i & -1+i\\ 0 & i \end{psmmatrix}$ is the $180^{\circ}$ rotation around the axis defined by the antipodal points $\begin{psmmatrix}1\\ 0\end{psmmatrix}$ and $\begin{psmmatrix}1\\ 1-i\end{psmmatrix}$, and
\item $f=\begin{psmmatrix}i & 1\\ i & 1-i\end{psmmatrix}$ is the rotation around the axis formed by connecting the centers of the opposite faces $\left\{\begin{psmmatrix}1\\ 1\end{psmmatrix},\begin{psmmatrix}1+i\\ 1\end{psmmatrix},\begin{psmmatrix}1\\ 0\end{psmmatrix}\right\}$ and $\left\{\begin{psmmatrix}0\\ 1\end{psmmatrix},\begin{psmmatrix}1\\ 1-i\end{psmmatrix},\begin{psmmatrix}i\\ 1\end{psmmatrix}\right\}$.
\end{itemize}

\begin{lemma}\label{lem:bstab}
The stabilizer $G_{\beta}$ of a block $\beta\in B$ in $G$ is isomorphic to the alternating group $A_4$.
\end{lemma}
\begin{proof}
First, since $G$ acts transitively on the block set $B$ it is enough to consider $G_T$, the stabilizer of the basic block $T$. In fact, we will explicitly write representatives for the elements of $G_T$. We will first show that $|G_T|=12$, and then one can verify that
$$
\left\{{\scriptsize
\begin{aligned}
&\begin{pmatrix}
1 & 0 \\
0 & 1 \\
\end{pmatrix},
\begin{pmatrix}
0 & -i \\
-i & -1 \\
\end{pmatrix},
\begin{pmatrix}
-1 & i \\
i & 0 \\
\end{pmatrix},
\begin{pmatrix}
-i & 0 \\
-1-i & i \\
\end{pmatrix},
\begin{pmatrix}
-1+i & 1 \\
i & -i \\
\end{pmatrix},
\begin{pmatrix}
1 & -1 \\
1 & 0 \\
\end{pmatrix},\\
&\begin{pmatrix}
1 & -1-i \\
1-i & -1 \\
\end{pmatrix},
\begin{pmatrix}
-1-i & i \\
-1 & i \\
\end{pmatrix},
\begin{pmatrix}
i & 1 \\
i & 1-i \\
\end{pmatrix},
\begin{pmatrix}
-i & -1+i \\
0 & i \\
\end{pmatrix},
\begin{pmatrix}
0 & 1 \\
-1 & 1 \\
\end{pmatrix},
\begin{pmatrix}
i & -i \\
1 & -1-i \\
\end{pmatrix}
\end{aligned}}\right\}
$$
is a set of $12$ representatives of distinct elements of $G_T$. We use the orbit-stabilizer theorem:
$$
\size{G_T}=\size{{G_T}_{\binom{1}{0}}}\cdot \size{\begin{psmmatrix}1\\ 0\end{psmmatrix}^{G_T}}.
$$
One can check directly that the elements in the above list leave $T$ fixed, and that there are elements in that list that take $\begin{psmmatrix}1\\ 0\end{psmmatrix}$ to each of the points in $T$, so $\size{\begin{psmmatrix}1\\ 0\end{psmmatrix}^{G_T}}=6$. To compute ${G_T}_{\binom{1}{0}}=G_T\cap G_{\binom{1}{0}}$ we recall that by~\autoref{lem:pstab}
$$
G_{\binom{1}{0}}=\set{\begin{psmmatrix}u & x\\ 0 & u^{-1}\end{psmmatrix}}{u\in\{1,i\},x\in\F_q},
$$
now it is enough to show that the only elements in $G_{\binom{1}{0}}$ that fix $T$ are $\begin{psmmatrix}1 & 0\\ 0 & 1\end{psmmatrix}$ and $\begin{psmmatrix}i & 1-i\\ 0 & -i\end{psmmatrix}$. So let $\begin{psmmatrix}u & x\\ 0 & u^{-1}\end{psmmatrix}$ be an element that stabilizes $T$, since $p\neq 5$ it follows from~\autoref{lem:pdiff5} that $\begin{psmmatrix}u & x\\ 0 & u^{-1}\end{psmmatrix}\cdot \begin{psmmatrix}0 \\1\end{psmmatrix}=\begin{psmmatrix}x\\ u^{-1}\end{psmmatrix}$ is either $\begin{psmmatrix}0\\ 1\end{psmmatrix}$, $\begin{psmmatrix}i\\ 1\end{psmmatrix}$, $\begin{psmmatrix}1\\ 1\end{psmmatrix}$ or $\begin{psmmatrix}1+i\\ 1\end{psmmatrix}$. Each of the above four cases gives $2$ options for $\begin{psmmatrix}u & x\\ 0 & u^{-1}\end{psmmatrix}$, and out of these eight options the only elements that indeed fix $T$ are $\begin{psmmatrix}1 & 0\\ 0 & 1\end{psmmatrix}$ and $\begin{psmmatrix}i & 1-i\\ 0 & -i\end{psmmatrix}$, so $\size{{G_T}_{\binom{1}{0}}}=2$ and $\size{G_T}=2\cdot 6=12$. The isomorphism $G_T\cong A_4$ can now be deduced by noticing that $G_T$ has no element of order $6$, and the only non-abelian group of order $12$ which has no elements of order $6$ is the alternating group $A_4$.
\end{proof}

\begin{lemma}\label{lem:edges}
\begin{enumerate}[(a)]
\item The group $G$ acts transitively on the edges of $D$.
\item Every edge of $D$ is contained in $4$ ocathedra.
\end{enumerate}
\end{lemma}
\begin{proof}
\begin{enumerate}[{\it (a)}]
\item Since $G$ acts transitively on the octahedra, it suffices to show that $G_T$ acts transitively on the $12$ edges of $T$. This will follow from the orbit-stabilizer theorem once we show that stabilizer ${G_T}_e$ of the edge $e=\left\{\begin{psmmatrix}1\\ 0\end{psmmatrix},\begin{psmmatrix}0\\ 1\end{psmmatrix}\right\}$ (in $G_T$) is trivial. So let $g\in {G_T}_e$, first we will show that $g$ must fix both vertices $\begin{psmmatrix}1\\ 0\end{psmmatrix}$ and $\begin{psmmatrix}0\\ 1\end{psmmatrix}$. Write $g=\begin{psmmatrix}a & b\\ c & d\end{psmmatrix}$, and assume that
$$
\begin{psmmatrix}a & b\\ c & d\end{psmmatrix}\cdot \begin{psmmatrix}1\\ 0\end{psmmatrix}=\begin{psmmatrix}0\\ 1\end{psmmatrix}\quad\text{and}\quad\begin{psmmatrix}a & b\\ c & d\end{psmmatrix}\cdot \begin{psmmatrix}0\\ 1\end{psmmatrix}=\begin{psmmatrix}1\\ 0\end{psmmatrix}
$$
then from the first equality it follows that $a=0$ and $c=1,i,-1$ or $-i$, from the second equality if follows that $d=0$ and $b=1,i,-1$ or $-i$, and by the determinant $=1$ condition we have
$$
\begin{psmmatrix}a & b\\ c & d\end{psmmatrix}=\begin{psmmatrix}0 & -1\\ 1 & 0\end{psmmatrix}\quad\text{or}\quad \begin{psmmatrix}a & b\\ c & d\end{psmmatrix}=\begin{psmmatrix}0 & i\\ i & 0\end{psmmatrix}.
$$
Both elements above are not in $G_T$, thus $g$ must fix each of the vertices of $e$. It follows that $g\in {G_T}_{\binom{1}{0}}$, thus $g=\begin{psmmatrix}1 & 0\\ 0 & 1\end{psmmatrix}$ or $g=\begin{psmmatrix}-i & -1+i\\ 0 & i\end{psmmatrix}$, but $\begin{psmmatrix}-i & -1+i\\ 0 & i\end{psmmatrix}$ does not fix $\begin{psmmatrix}0\\ 1\end{psmmatrix}$ so we must have $g=\begin{psmmatrix}1 & 0\\ 0 & 1\end{psmmatrix}$ and so ${G_T}_e$ is trivial.
\item From the calculation in the proof of $(a)$ it follows that 
$$
G_e=\left\langle
\begin{psmmatrix}0 & -1\\ 1 & 0\end{psmmatrix},
\begin{psmmatrix}0 & i\\ i & 0\end{psmmatrix}
\right\rangle
$$
thus $\size{G_e}=4$. Using the orbit-stabilizer theorem we can now compute the total number of edges of $D$, that is, the size of the orbit of $e$ under the action of $G$:
$$
\size{e^G}=\frac{\size{G}}{\size{G_e}}=\frac{\frac{q(q^2-1)}{2}}{4}=\frac{q(q^2-1)}{8}.
$$
Let $E$ denote the set of edges of $D$. Now we apply a standard double-counting argument, we count the size of the set
$$
S=\set{(\epsilon,\beta)}{\epsilon\in E, \beta\in B, \epsilon \text{ is an edge of }\beta}
$$
in two ways:
\begin{itemize}
\item $\size{S}=12\cdot\size{B}$ because every block has $12$ edges, and
\item $\size{S}=\size{E}\cdot x$, where $x$ is the number of blocks that contain a given edge.
\end{itemize}
Therefore, we have
$$
x=\frac{12\cdot\size{B}}{\size{E}}=\frac{12\cdot\frac{q(q^2-1)}{24}}{\frac{q(q^2-1)}{8}}=4.
$$
\end{enumerate}
\end{proof}

\begin{lemma}\label{lem:diagonals}
\begin{enumerate}[(a)]
\item The group $G$ acts transitively on the diagonals of $D$.
\item Every diagonal of $D$ is contained in $1$ ocathedron.
\end{enumerate}
\end{lemma}
\begin{proof}
\begin{enumerate}[{\it (a)}]
\item As in the previous proofs, since $G$ acts transitively on the octahedra, it suffices to show that $G_T$ acts transitively on the $3$ diagonals of $T$. To see that, it is enough to consider just the rotation $f$:
$$
\left\{\begin{psmmatrix}1\\ 0\end{psmmatrix},\begin{psmmatrix}1\\ 1-i\end{psmmatrix}\right\}\overset{f}{\mapsto} \left\{\begin{psmmatrix}1\\ 1\end{psmmatrix},\begin{psmmatrix}i\\ 1\end{psmmatrix}\right\}\overset{f}{\mapsto}
\left\{\begin{psmmatrix}1+i\\ 1\end{psmmatrix},\begin{psmmatrix}0\\ 1\end{psmmatrix}\right\}.
$$
\item The following arguments are similar to those given in the proof of~\autoref{lem:edges}, we will compute the stabilizer $G_d$ of the diagonal $d=\left\{\begin{psmmatrix}1\\ 0\end{psmmatrix},\begin{psmmatrix}1\\ 1-i\end{psmmatrix}\right\}$ in order to obtain the total number of diagonals in $D$ via the orbit-stabilizer theorem, and then use a double-counting argument to compute the number of octahedra that contain a given diagonal. So let $g=\begin{psmmatrix}a & b\\ c & d\end{psmmatrix}\in G_d$, then we have either
$$
(I)\quad
\begin{psmmatrix}a & b\\ c & d\end{psmmatrix}\cdot \begin{psmmatrix}1\\ 0\end{psmmatrix}=\begin{psmmatrix}1\\ 0\end{psmmatrix}\quad\text{and}\quad\begin{psmmatrix}a & b\\ c & d\end{psmmatrix}\cdot \begin{psmmatrix}1\\ 1-i\end{psmmatrix}=\begin{psmmatrix}1\\ 1-i\end{psmmatrix}
$$
or
$$
(II)\quad
\begin{psmmatrix}a & b\\ c & d\end{psmmatrix}\cdot \begin{psmmatrix}1\\ 0\end{psmmatrix}=\begin{psmmatrix}1\\ 1-i\end{psmmatrix}\quad\text{and}\quad\begin{psmmatrix}a & b\\ c & d\end{psmmatrix}\cdot \begin{psmmatrix}1\\ 1-i\end{psmmatrix}=\begin{psmmatrix}1\\ 0\end{psmmatrix}.
$$

In case $(I)$, the first equality implies $a=1$ or $i$, and $c=0$, thus $\begin{psmmatrix}a & b\\ c & d\end{psmmatrix}=\begin{psmmatrix}1 & b\\ 0 & 1\end{psmmatrix}$ or $\begin{psmmatrix}a & b\\ c & d\end{psmmatrix}=\begin{psmmatrix}i & b\\ 0 & -i\end{psmmatrix}$, plugging this into the second equation we get $b=0$ in the first case and $b=1-i$ in the second case, so $$\begin{psmmatrix}a & b\\ c & d\end{psmmatrix}=\begin{psmmatrix}1 & 0\\ 0 & 1\end{psmmatrix}\text{ or }\begin{psmmatrix}a & b\\ c & d\end{psmmatrix}=\begin{psmmatrix}i & 1-i\\ 0 & -i\end{psmmatrix}.$$

In case $(II)$, the first equality implies $a=1$ or $i$, and $c=1-i$ or $1+i$ (respectively), thus $\begin{psmmatrix}a & b\\ c & d\end{psmmatrix}=\begin{psmmatrix}1 & b\\ 1-i & d\end{psmmatrix}$ or $\begin{psmmatrix}a & b\\ c & d\end{psmmatrix}=\begin{psmmatrix}i & b\\ 1+i & d\end{psmmatrix}$, plugging this into the second equation we get $d=-1$ in the first case, and $d=-i$ in the second, so $\begin{psmmatrix}a & b\\ c & d\end{psmmatrix}=\begin{psmmatrix}1 & b\\ 1-i & -1\end{psmmatrix}$ or $\begin{psmmatrix}a & b\\ c & d\end{psmmatrix}=\begin{psmmatrix}i & b\\ 1+i & -i\end{psmmatrix}$, now the determinant $=1$ condition determines $b$ and we get
$$\begin{psmmatrix}a & b\\ c & d\end{psmmatrix}=\begin{psmmatrix}1 & -1-i\\ 1-i & -1\end{psmmatrix}\text{ or }\begin{psmmatrix}a & b\\ c & d\end{psmmatrix}=\begin{psmmatrix}i & 0\\ 1+i & -i\end{psmmatrix}.$$
It is now easy to check that the four representative matrices we found form a group, thus $\size{G_d}=4$. Again, the orbit-stabilizer theorem allows us to compute the total number of diagonals of $D$, that is, the size of the orbit of $d$ under the action of $G$:
$$
\size{d^G}=\frac{\size{G}}{\size{G_d}}=\frac{\frac{q(q^2-1)}{2}}{4}=\frac{q(q^2-1)}{8}.
$$
Let $\Delta$ denote the set of diagonals of $D$. Yet again we apply a double-counting argument, we count the size of the set
$$
S=\set{(\delta,\beta)}{\delta\in\Delta, \beta\in B, \delta \text{ is a diagonal of }\beta}
$$
in two ways:
\begin{itemize}
\item $\size{S}=3\cdot\size{B}$ because every block has $3$ diagonals, and
\item $\size{S}=\size{\Delta}\cdot x$, where $x$ is the number of blocks that contain a given diagonal.
\end{itemize}
Therefore, we have
$$
x=\frac{3\cdot\size{B}}{\size{\Delta}}=\frac{3\cdot\frac{q(q^2-1)}{24}}{\frac{q(q^2-1)}{8}}=1.
$$
\end{enumerate}
\end{proof}

\begin{theorem}\label{thm:pbibd}
The block design $D$ is a $\gpbibd(m)$ with parameters 
$$
v=\frac{q^2-1}{4},\quad b=\frac{q(q^2-1)}{24},\quad k=6,\quad r=q,\quad \lambda_1=4,\quad \lambda_2=1,\quad \lambda_3=\dots=\lambda_m=0.
$$
The underlying association scheme is the Schurian scheme $\cS(G,P)$ of $(G,P)$, and the number of associate classes is $$m=\frac{q-3}{2}.$$
\end{theorem}
\begin{proof}
The number of points is clear from the construction. The number of blocks can be calculated by using the orbit-stabilizer theorem and the fact that, by~\autoref{lem:bstab}, $\size{G_T}=12$: $$b=\size{B}=\size{T^G}=\frac{\size{G}}{\size{G_T}}=\frac{\frac{q(q^2-1)}{2}}{12}=\frac{q(q^2-1)}{24}.$$
The fact that each block has size $k=6$ follows from the fact that the basic block $T$ has size $6$, and that the set of blocks is the orbit $T^G$ of $T$ under the action of $G$. To prove that each point lies on $r=q$ blocks it suffices to show that $\begin{psmmatrix}1\\ 0\end{psmmatrix}$ lies on $q$ blocks, because $G$ acts transitively on $P$. For that purpose we will show that $\size{T^{G_{\binom{1}{0}}}}=q$. Again, the orbit-stabilizer theorem applied to $G_{\binom{1}{0}}$ will do the job: 
$$2q=\size{G_{\binom{1}{0}}}=\size{{G_{\binom{1}{0}}}_T}\cdot \size{T^{G_{\binom{1}{0}}}}=2\cdot \size{T^{G_{\binom{1}{0}}}},$$ 
thus $\size{T^{G_{\binom{1}{0}}}}=q$. A pair of points of $P$ is either an edge of $D$, a diagonal of $D$ or neither (that is a pair that does not lie in a block together), it follows from~\autoref{lem:edges} that an edge lies on $\lambda_1=4$ blocks, from~\autoref{lem:diagonals} that a diagonal lies on $\lambda_2=1$ block, and any other pair lies on $\lambda_j=0$ ($3\leq j\leq m$) blocks. Next, the fact that the Schurian scheme of the action $(G,P)$ is the underlying association scheme follows from the fact that $G$ acts transitively on the edges of $D$ (proved in~\autoref{lem:edges}), and on the diagonals of $D$ (proved in~\autoref{lem:diagonals}). 

Finally, we compute the number of associate classes, that is, the number of classes of $\cS(G,P)$, or in other words, the rank of the transitive permutation group $(G,P)$ minus $1$. The rank of a transitive permutation group is equal to number of orbits of the stabilizer of a point, so we compute the number of orbits of $\left(G_{\binom{1}{0}},P\right)$. Recall that 
$$ G_{\binom{1}{0}}=\set{\begin{psmmatrix}u & x\\ 0 & u^{-1}\end{psmmatrix}}{u\in\{1,i\},x\in\F_q},$$ 
so for each point of the form $\begin{psmmatrix}a\\0\end{psmmatrix}\in P$ we have $\begin{psmmatrix}a\\0\end{psmmatrix}^{G_{\binom{1}{0}}}=\left\{\begin{psmmatrix}a\\0\end{psmmatrix}\right\}$, these are $\frac{q-1}{4}$ orbits. For points of the form $\begin{psmmatrix}a\\b\end{psmmatrix}\in P$, with $b\neq 0$, we have
$$
\begin{psmmatrix}a\\b\end{psmmatrix}^{G_{\binom{1}{0}}}=\set{\begin{psmmatrix}y\\b\end{psmmatrix}}{y\in\F_q},
$$
because $\begin{psmmatrix}1 & \frac{y-a}{b}\\ 0 & 1\end{psmmatrix}$ takes $\begin{psmmatrix}a\\b\end{psmmatrix}$ to $\begin{psmmatrix}y\\b\end{psmmatrix}$, so we get $\frac{q-1}{4}$ more orbits. Altogether we counted the $\frac{q-1}{2}$ orbits of $G_{\binom{1}{0}}$, thus the number of associate classes is $m=\frac{q-3}{2}$.
\end{proof}

\section*{The case $p=5$}\label{sec:pe5}

Let $q=5^{\alpha}$. In this case the basic block $T$ is in fact just the projective line over $\F_5$, and we have

\begin{lemma}\label{lem:bstab5}
The stabilizer $G_{\beta}$ of a block $\beta\in B$ in $G$ is isomorphic to $\psl(2,5)$.
\end{lemma}
\begin{proof}
As before, by the transitive action of $G$ on $B$ it is enough to consider $G_T$. By~\autoref{lem:pdiff5}, since $q=5^{\alpha}$, we have that $1+i=-i$. It follows that $1-i=-1$ and thus we have
$$T=\left\{\begin{psmmatrix}1\\ 0\end{psmmatrix},\begin{psmmatrix}0\\ 1\end{psmmatrix},\begin{psmmatrix}1\\ 1\end{psmmatrix},\begin{psmmatrix}-i\\ 1\end{psmmatrix},\begin{psmmatrix}i\\ 1\end{psmmatrix},\begin{psmmatrix}-1\\ 1\end{psmmatrix}\right\},$$
that is, $T$ is the projective line over $\F_5$, and it follows that subgroup of $\psl(2,q)$ that fixes the projective line over $\F_5$ is $\psl(2,5)$. 
\end{proof}

\begin{corollary}\label{cor:trans5}
$G$ acts transitively on pairs of points that lie together in a block.
\end{corollary}

We will refer to pairs of points that lie together on a block {\em adjacent} points, and correspondingly to pairs of points that do not lie together on a block as {\em non-adjacent} points.

\begin{theorem}\label{thm:pbibd5}
The block design $D$ is a $\gpbibd(m)$ with parameters 
$$
v=\frac{q^2-1}{4},\quad b=\frac{q(q^2-1)}{120},\quad k=6,\quad r=\frac{q}{5}=5^{\alpha-1},\quad \lambda_1=1,\quad \lambda_2=\dots=\lambda_m=0.
$$
The underlying association scheme is the Schurian scheme $\cS(G,P)$ of $(G,P)$, and the number of associate classes is $$m=\frac{q-3}{2}.$$
\end{theorem}
\begin{proof}
The number of points is clear from the construction. The number of blocks can be calculated by using the orbit-stabilizer theorem and the fact that, by~\autoref{lem:bstab5}, $\size{G_T}=60$: $$b=\size{B}=\size{T^G}=\frac{\size{G}}{\size{G_T}}=\frac{\frac{q(q^2-1)}{2}}{60}=\frac{q(q^2-1)}{120}.$$
The fact that each block has size $k=6$ follows from the fact that the basic block $T$ has size $6$, and that the set of blocks is the orbit $T^G$ of $T$ under the action of $G$. To prove that each point lies on $r=\frac{q}{5}=5^{\alpha-1}$ blocks it suffices to show that $\begin{psmmatrix}1\\ 0\end{psmmatrix}$ lies on $5^{\alpha-1}$ blocks, because $G$ acts transitively on $P$. For that purpose we will show that $\size{T^{G_{\binom{1}{0}}}}=5^{\alpha-1}$. Again, the orbit-stabilizer theorem applied to $G_{\binom{1}{0}}$ will do the job: 
$$2q=\size{G_{\binom{1}{0}}}=\size{{G_{\binom{1}{0}}}_T}\cdot \size{T^{G_{\binom{1}{0}}}}.$$
By~\autoref{lem:bstab5} we have
$${G_{\binom{1}{0}}}_T=G_T\cap G_{\binom{1}{0}}=\psl(2,5)\cap G_{\binom{1}{0}},$$
that is, ${G_{\binom{1}{0}}}_T=\set{\begin{psmmatrix}u & x\\ 0 & u^{-1}\end{psmmatrix}}{u\in\{1,i\},x\in\F_5}$ and $\size{{G_{\binom{1}{0}}}_T}=10$, thus $2q=10\cdot\size{T^{G_{\binom{1}{0}}}}$ and therefore $\size{T^{G_{\binom{1}{0}}}}=\frac{2q}{10}=\frac{2\cdot 5^{\alpha}}{2\cdot 5}=5^{\alpha-1}$. The number of pairs of adjacent points is $$\binom{6}{2}\cdot\size{B}=15\cdot\frac{q(q^2-1)}{120}=\frac{q(q^2-1)}{8}.$$ To show that each pair of adjacent points determines a unique block it is enough to show that the stabilizer of a pair of adjacent points stabilizes that whole block. So by~\autoref{cor:trans5} we may assume that $e=\left\{\begin{psmmatrix}1\\ 0\end{psmmatrix},\begin{psmmatrix}0\\ 1\end{psmmatrix}\right\}$ is our pair of adjacent points on the block $T$. We need to show that $G_e\leq G_T$. Recall that the calculation in~\autoref{lem:edges} showed that 
$$
G_e=\left\langle
\begin{psmmatrix}0 & -1\\ 1 & 0\end{psmmatrix},
\begin{psmmatrix}0 & i\\ i & 0\end{psmmatrix}
\right\rangle,
$$
in this case, clearly $G_e\leq\psl(2,5)=G_T$. Finally, the fact that the Schurian scheme of the action $(G,P)$ is the underlying association scheme follows from the fact that $G$ acts transitively on the pairs of adjacent points $D$ (stated in~\autoref{cor:trans5}), the computation of the number of associate classes is identical to that given in~\autoref{thm:pbibd}.
\end{proof}

At this point it seems interesting to ask:

\begin{question}
Can we find an underlying association scheme for $D$ with less vanishing classes?
\end{question}

The following two sections deal with this question. In~\autoref{sec:larger_group} we find the largest (in a certain natural sense) group acting on $D$ producing an underlying Schurian association scheme with less vanishing classes. In~\autoref{sec:lambda_partition_stabilization} we compute the underlying association scheme with the smallest possible number of vanishing classes for all $\gpbibd$s with $q\leq 169$, and propose a general conjecture about the minimal association schemes.

\section{A smaller rank underlying association scheme}\label{sec:larger_group}

In this section we prove that in fact we have a larger group acting on $D$. Let $\phi$ be the Frobenius automorphism of $\F_q$, that is, the map that takes each element $a\in\F_q$ to its $p$-th power $a^p\in\F_q$. We denote by $F=\langle \phi\rangle$ the automorphism group $\aut(\F_q)$. The projective special semilinear group, commonly denoted by $\psil(2,q)$, is the semidirect product $\psl(2,q)\rtimes F$, where the cyclic group $F$ is acting on $\psl(2,q)$ in the natural way: 
$$\begin{psmmatrix}a & b\\ c & d\end{psmmatrix}^{\phi}=\begin{psmmatrix}a^p & b^p\\ c^p & d^p\end{psmmatrix}.$$

\begin{theorem}\label{thm:full_aut}
The largest automorphism group $\cG$ of $D$ that preserves the octahedral structure is isomorphic to $\psil(2,q)\times C_2$.
\end{theorem}
\begin{proof}
Clearly $\psil(2,q)$ acts on $D$. The additional involution $C_2=\langle \sigma\rangle$ is the missing $90^{\circ}$ rotation about the axis formed by two antipodal vertices of the octahedron. 
\end{proof}

\subsection*{Computing the number of orbits of $F$ on $P$}

We use Burnside's lemma and standard Galois cohomology tools to compute the number of orbits $\bigslant{P}{F}$ of $F$ on $P=\bigslant{{\F_q}^{\times}}{\mu_4}$ and $q=p^n$. As usual we denote by $P^g$ the fixed points of $P$ under the action of $g$. Burnside's lemma establishes that
$$
\left\vert\bigslant{P}{F}\right\vert=\frac{1}{\vert F\vert}\sum_{g\in F} \left\vert P^g\right\vert.
$$
For each divisor $d$ of $n$ we have a subfield $\F_{p^d}$ of $\F_{p^n}$, the elements of $\F_{p^d}$ are precisely the fixed points of $\phi^d$. Moreover, if $k$ is a multiple of $d$ that does not divide $n$ then the fixed points of $\phi^k$ are again the elements of $\F_{p^d}$, thus we can rewrite the above sum as follows:
$$
\left\vert\bigslant{P}{F}\right\vert=\frac{1}{\vert F\vert}\sum_{g\in F} \left\vert P^g\right\vert=\frac{1}{\vert F\vert}\sum_{d\vert n}\varphi\left(\frac{n}{d}\right) \left\vert P^{\phi^d}\right\vert
$$  
where $\varphi$ is Euler totient function. To compute $\left\vert P^{\phi^d}\right\vert$ we use Galois cohomology. For clearer notation we denote $K:=\phi^d$. The short exact sequence
$$
1\lra \mu_4\lra {\F_q}^{\times}\lra \bigslant{{\F_q}^{\times}}{\mu_4}\lra 1
$$
provides a long exact sequence in group cohomology
$$
1\lra {\mu_4}^K\lra \left({\F_q}^{\times}\right)^K\lra \left(\bigslant{{\F_q}^{\times}}{\mu_4}\right)^K\lra H^1(K,\mu_4)\lra H^1(K,{\F_q}^{\times}) \lra\dots
$$
By Hilbert's Theorem 90 we have 
$$
H^1(K,{\F_q}^{\times})=1
$$
so we get the following exact sequence
$$
1\lra {\mu_4}^K\lra \left({\F_q}^{\times}\right)^K\lra \left(\bigslant{{\F_q}^{\times}}{\mu_4}\right)^K\lra H^1(K,\mu_4)\lra  1
$$
and we are interested in 
$$\left\vert\left(\bigslant{{\F_q}^{\times}}{\mu_4}\right)^K\right\vert=\frac{\left\vert H^1(K,\mu_4)\right\vert\cdot\left\vert \left({\F_q}^{\times}\right)^K\right\vert}{\left\vert {\mu_4}^K\right\vert}.$$
Clearly, we have $\left\vert \left({\F_q}^{\times}\right)^K\right\vert=p^d-1$. We deal with cases $p\equiv 1\pmod 4$ and $p\equiv 3\pmod 4$ separately.
\subsection*{The case $p\equiv 1\pmod 4$} In this case $\mu_4$ is already contained in $\F_p$ and thus is fixed by $K$ so $\left\vert {\mu_4}^K\right\vert=4$ and we have
$$\left\vert\left(\bigslant{{\F_q}^{\times}}{\mu_4}\right)^K\right\vert=\left\vert H^1(K,\mu_4)\right\vert\cdot\frac{p^d-1}{4}.$$
Now, since $\mu_4$ is a trivial $K$-module we have
$$
H^1(K,\mu_4)=\Hom(K,\mu_4)\cong\begin{cases} 
C_4 & \vert K\vert\equiv 0\pmod 4\\
C_2 & \vert K\vert\equiv 2\pmod 4\\
1 & \text{otherwise}
\end{cases}
$$
so if we denote $h(d):=\vert H^1(K,\mu_4)\vert$ (this number depends on the residue of $\frac{n}{d}$ modulo $4$) we have
$$
\size{\bigslant{P}{F}}=\frac{1}{\size{F}}\sum_{d\vert n}\varphi\left(\frac{n}{d}\right)\cdot h(d)\cdot\frac{p^d-1}{4}.
$$

\subsection*{The case $p\equiv 3\pmod 4$} In this case $\mu_4$ is contained in $\F_{p^2}$ but not in $\F_P$ so for every odd power of $\phi$ there are just $2$ fixed points, and for every even power of $\phi$ there are $4$ fixed points, that is:
$$
\size{{\mu_4}^K}=\begin{cases}
2 & d\text{ odd}\\
4 & d\text{ even}
\end{cases}.
$$
The computation of $\size{H^1(K,\mu_4)}$ in this case is a bit more involved. If $\frac{n}{d}$ is even then $\mu_4$ is again a trivial $K$-module and we have
$$
H^1(K,\mu_4)=\Hom(K,\mu_4)\cong\begin{cases} 
C_4 & \vert K\vert\equiv 0\pmod 4\\
C_2 & \vert K\vert\equiv 2\pmod 4\\
\end{cases}.
$$
For $d$ odd, $\mu_4$ is no longer a trivial $K$-module. There are two cyclic groups here, and an action of one of them on the other, considered as a module, so we will use additive notation for $\mu_4$ (that is, the group of integers modulo $4$), and write $a\cdot m$ to denote the action of $a$ on $m$, where $a\in K$ and $m\in\mu_4$, the operation in $K$ will be denoted simply $ab$ for $a,b\in K$. Fix a generator $\alpha$ of $K$, the the action of $K$ on $\mu_4$ is
$$
\alpha\cdot 0=0,\quad \alpha\cdot 1=3,\quad \alpha\cdot 2=2,\quad \alpha\cdot 3=1.
$$
The $1$-cocycles $Z^1(K,\mu_4)$ are thus the functions $f:K\longrightarrow\mu_4$ satisfying
$$
f(ab)=f(a)+a\cdot f(b) 
$$
and the $1$-coboundaries $B^1(K,\mu_4)$ are those satisfying
$$
f(a)=a\cdot m -m\qquad\text{for some}\quad m\in\mu_4.
$$
A $1$-cocycle is determined by its value on $\alpha$. So let $f\in B^1(K,\mu_4)$, one can verify that out of the $4$ options (one option for each $m\in\mu_4$) there are just $2$ elements in $B^1(K,\mu_4)$, explicitly, these are the functions
$$
f(\alpha)=\alpha\cdot 0 -0\quad\text{and}\quad f(\alpha)=\alpha\cdot 1 -1.
$$
Consider now $Z^1(K,\mu_4)$. Let $f\in Z^1(K,\mu_4)$ then out of the $4$ options
$$
f(\alpha)=0,\quad f(\alpha)=1,\quad f(\alpha)=2,\quad f(\alpha)=3,
$$
it is easy to verify that $f(\alpha)=0$ and $f(\alpha)=2$ are both in $B^1(K,\mu_4)$. Also, $f(\alpha)=1$ and $f(\alpha)=3$ are cohomologous (their difference is in $B^1(K,\mu_4)$), so altogether we get
$$
H^1(K,\mu_4)=\bigslant{Z^1(K,\mu_4)}{B^1(K,\mu_4)}\cong C_2.
$$
Denote $h(d):=\frac{\size{H^1(K,\mu_4)}}{\size{{\mu_4}^K}}$ summarizing the above we have
$$
h(d)=\begin{cases}
1 & \text{$d$ odd; or $d$ even and $\frac{n}{d}\equiv 0\pmod 4$}\\
\frac{1}{2} & \text{$d$ even and $\frac{n}{d}\equiv 2\pmod 4$}\\
\frac{1}{4} & \text{$d$ even and $\frac{n}{d}$ odd}
\end{cases}
$$
and
$$
\size{\bigslant{P}{F}}=\frac{1}{\size{F}}\sum_{d\vert n}\varphi\left(\frac{n}{d}\right)\cdot h(d)\cdot(p^d-1).
$$

\begin{corollary}\label{cor:smallest_rank}
$D$ is a $\gpbibd(m)$ with $m=2\cdot\size{\bigslant{P}{F}}-1$ associate classes.
\end{corollary}

\section{The smallest rank underlying association scheme}\label{sec:lambda_partition_stabilization}

We have constructed all the $\gpbibd$s with $q\leq 169$ on {\sf GAP}~\cite{GAP4}, we used the program {\sf stabil2} to perform WL-stabilization on the partition of $P\times P$ into $\lambda$-classes (a $\lambda$-class consists of all pairs in $P\times P$ that lie on $\lambda$ blocks) in order to obtain an underlying association scheme with the smallest number of classes possible for each incidence structure. It turns out that in a considerable number of cases we in fact obtain an underlying scheme of smaller rank than what is provided by~\autoref{cor:smallest_rank}. It follows from~\autoref{thm:full_aut} that in those cases the underlying scheme is non-Schurian.
\newpage
%
%
%
%
%
%

\begin{center}
\begin{longtable}{ | c | c | c | c | m{4cm} |}
\hline
\thead{$q$} & \thead{Class\\ Number in\\ \autoref{cor:smallest_rank}} & \thead{Smallest\\ Class\\ Number} & \thead{Parameters\\ $(v,b,r,k;\lambda_1,\dots ,\lambda_m)$} & \thead{Remarks} \\ 
\hline\hline
\hline
$9$ & $3$ & $3$ & $(20,30,6,9;4,1,0)$ & \\
\hline
$13$ & $5$ & $5$ & $(42,91,6,13;4,1,0,0,0)$ & \\
\hline
$17$ & $7$ & $7$ & $(72,204,6,17;4,1,0,\dots,0)$ & \\
\hline
$25$ & $7$ & $3$ & $(156,130,6,5;1,0,0)$ & Non-Schurian. This is an antipodal distance regular graph of diameter $3$, a $6$-fold cover of $K_4$.\\
\hline
$29$ & $13$ & $13$ & $(210,1015,6,29;4,1,0,\dots,0)$ & \\
\hline
$37$ & $17$ & $17$ & $(342,2109,6,37;4,1,0,\dots,0)$ & \\
\hline
$41$ & $19$ & $11$ & $(420,2870,6,41;4,1,0,\dots,0)$ & Non-Schurian. Non-commutative.\\
\hline
$49$ & $17$ & $13$ & $(600,4900,6,49;4,1,0,\dots,0)$ & Non-Schurian. Non-commutative.\\
\hline
$53$ & $25$ & $25$ & $(702,6201,6,53;4,1,0,\dots,0)$ & \\
\hline
$61$ & $29$ & $29$ & $(930,9455,6,61;4,1,0,\dots,0)$ & \\
\hline
$73$ & $35$ & $35$ & $(1332,16206,6,73;4,1,0,\dots,0)$ & \\
\hline
$81$ & $13$ & $5$ & $(1640,22140,6,81;4,1,0,0,0)$  & non-Schurian.\\
\hline
$89$ & $43$ & $43$ & $(1980,29370,6,89;4,1,0,\dots,0)$ & \\
\hline
$97$ & $47$ & $47$ & $(2352,38024,6,97;4,1,0,\dots,0)$ & \\
\hline
$101$ & $49$ & $49$ & $(2550,42925,6,101;4,1,0,\dots,0)$ & \\
\hline
$109$ & $53$ & $19$ & $(2970,53955,6,109;4,1,0,\dots,0)$ & Non-Schurian. Non-commutative. \\
\hline
$113$ & $55$ & $15$ & $(3192,60116,6,113;4,1,0,\dots,0)$ & Non-Schurian. Non-commutative. \\
\hline
$121$ & $39$ & $21$ & $(3660,73810,6,121;4,1,0,\dots,0)$ & non-Schurian. \\
\hline
$125$ &  & $3$ & $(3906,16275,6,25;1,0,\dots,0)$ & non-Schurian. This is an antipodal distance regular graph of diameter $3$, a $31$-fold cover of $K_4$.\\
\hline
$137$ & & & & \\
\hline
$149$ & & & & \\
\hline
$157$ & & & & \\
\hline
$169$ & $47$ & $7$ & $(7140,201110,6,169;4,1,0,\dots,0)$ & non-Schurian. \\
\hline
\end{longtable}
\end{center}

\subsection*{Acknowledgements}
I thank Rosemary Bailey and Eran Nevo for helpful comments on earlier versions of this text.

\bibliographystyle{plain}
\bibliography{gpbibds}

\end{document}